\documentclass[12pt,a4]{article}

\usepackage{url}

\usepackage{color}
\usepackage[T1]{fontenc}
\usepackage[english]{babel}
\usepackage{setspace}
\usepackage{csquotes}

\usepackage[margin=1in]{geometry} 
\usepackage{amsmath,amsthm,amssymb,mathrsfs}
\usepackage{kpfonts}
\usepackage{xcolor}
\usepackage{hyperref}
\hypersetup{
    colorlinks=true,
    linkcolor=blue,
    filecolor=magenta,
    urlcolor=cyan,
    citecolor=blue,
}
\usepackage{enumitem}
\usepackage{graphicx}
\usepackage{subcaption}
\usepackage{float}

\newcommand{\Prob}{\mathbb{P}}
\newcommand{\Esp}{\mathbb{E}}

\newcommand{\Z}{\mathbb{Z}}
\newcommand{\R}{\mathbb{R}}

\newtheorem{df}{Definition}[section]

\newtheorem{pro}[df]{Proposition}
\newtheorem{thm}[df]{Theorem}

 \usepackage{fancyhdr}
\usepackage{atbegshi}

\addto\captionsfrench{}
\title{ Asymptotic behavior of some stochastic models in population dynamics: a Hamilton-Jacobi approach}
\begin{document}

\author{\large{Anouar Jedd$^{\ast}$}\\
\small{$^{\ast}$CMAP, École polytechnique, Institut polytechnique de Paris,}\\
    \text{anouar.jeddi@polytechnique.edu}
}

\maketitle
\begin{abstract}
   In this paper, we investigate the asymptotic behavior of individual-based models describing the evolution of a population structured by a real trait, subject to selection and mutation. We consider two different sets of assumptions: first, the case of critical or subcritical branching population processes in a regime combining a discretization of the trait space, small mutations, large time and large initial population size, where we are able to characterize using a Hamilton-Jacobi approach, the survival set of the population, and the asymptotic of the logarithmic scaling of subpopulation sizes. Second, we generalize by a direct method the convergence to the classical Hamilton-Jacobi equation obtained in the super-critical branching regime considered in \cite{CMMT} to a more general trait space and under weaker assumptions. Moreover, we establish that the stochastic and the deterministic dynamics are asymptotically equivalent in large population. 
\end{abstract}
\section{Introduction}
In this paper, we provide a probabilistic justification for certain Hamilton-Jacobi equations from stochastic individual-based models in a large population regime. Our goal is to prove that, on  logarithmic time and size scales, the sizes of subpopulations converge to a viscosity solution of a Hamilton-Jacobi equation, able to account for extinction of subpopulations.\\

In recent years, a Hamilton-Jacobi approach has been emerged as an important tool in the mathematical analysis of evolutionary biology (see \cite{OPSB}). This approach consists in studying the evolution of populations structured by a quantitative trait subject to selection and mutation, in a regime of small mutations and large time and considering exponential orders of magnitude of population sizes.  The original approach is based on PDEs models in population dynamics, structured by phenotypes, or traits (see for instance \cite{OPSB,PB,BMP,J}). More precisely, these references prove that the Hopf-Cole transformation of the population density converges to a viscosity solution of a Hamilton-Jacobi equation. This approach captures the concentration of the population into Dirac masses at the dominant traits, characterized through the associated viscosity solution. This deterministic approach takes into account the evolution of negligible populations, but not the possibility of local extinction of subpopulations, since population densities in the PDE are always positive everywhere in space. In this context, some attempts to account for extinct populations have been introduced in \cite{MPBS} (see also \cite{Jabin}), by imposing a singular mortality rate below a survival threshold, and in this case the limit dynamics is not anymore described by a classical Hamilton-Jacobi equation.\\

A natural alternative approach to account for extinctions would be to consider stochastic individual based models instead PDEs models. This approach was followed in several references \cite{CFM, 2CFM,CMT,CKS}. Classically the link between individual based models and Hamilton-Jacobi equations was carried out in two steps: first, prove the convergence of  the individual based model to a PDE in a limit of large population (see \cite{CFM,2CFM}); second obtain the Hamilton-Jacobi equation from the PDE as explained above. However, these two successive scalings   do not allow to account for possible  extinction of subpopulations, as discussed above.\\

In \cite{CMT,CKS} individual based models with rare mutations were considered in a discrete trait space and with a scaling parameter $K.$ Defining the subpopulation size of individuals of type  $i$ at time $t$ as $N^K_i(t),$ these references establish an asymptotic behavior of the exponents $\beta^K_i(t)$ defined by $N^K_i(t\log{K})=K^{\beta^K_i(t)}$. These exponents can be seen as a Hopf-Cole transform in a discrete  trait space with logarithmic time scaling. In these references a regime of rare mutations was considered, but not small mutations.\\

In order to model small mutations, the trait space should be discretized in steps depending on the mutation sizes. Recently in \cite{CMMT} such a discretization of the trait space was introduced by considering a parameter $\delta_K\rightarrow 0$ as $K\rightarrow+\infty$. This work introduces a  birth-death individual based model with mutation in the discrete trait space with small mutations of size of the order $1/\log{K}.$  In the limit $K\rightarrow+\infty,$ the exponents $\beta^K_i$ defined as above by the Hopf-Cole transform  are proved to converge to a viscosity solution of a Hamilton-Jacobi equation. Up to our knowledge, this is the first and only work which establishes a  direct derivation of a Hamilton-Jacobi equation from an individual-based model, but under restrictive assumptions. In particular their branching model does not allow for extinction, because the reproduction is assumed uniformly super-critical in space (the birth rate is strictly greater than the death rate),  and the initial subpopulation sizes are lower bounded. This implies the survival of all subpopulations at all time with probability converging to $1.$ In addition this reference considers only the case of the one dimensional torus as trait space. The method used relies on compactness and Lipschitz estimates  requiring these assumptions.\\

Our objective in this work is to extend the derivation  of Hamilton-Jacobi equations from individual-based models in two more general settings. The first one allows local extinctions and the second one generalizes the super-critical case.
\paragraph{First part:} In this part we study  the case of critical or subcritical spatial branching processes in the real trait space. We assume no lower bound on the initial subpopulation sizes so that extinction of subpopulations and ultimately extinction of the whole population occur almost surely. Under the classical Hopf-Cole transform, the extinction of a subpopulation means that the exponent takes the value $-\infty.$ Therefore, we cannot expect anymore the tightness of the exponents to hold at all times and across the entire trait space, and the limit Hamilton-Jacobi equation needs to be modified to account for extinction. Then the question consists in characterizing the limit Hamilton-Jacobi equation by introducing appropriate cut-off and in proving the convergence of exponents of subpopulations. Following the same intuition as in \cite{MPBS}, we expect to obtain the Hamilton-Jacobi equation in a region of the spatio-temporal space corresponding to the survival set and $-\infty$ for extinct subpopulations. In the critical or subcritical case, we will prove that this equation  actually takes a very simple form (much simpler than in \cite{MPBS}), obtained by directly applying a cut-off to the original Hamilton-Jacobi equation when it is negative. Note that the convergence of the exponents to $-\infty$ does not make use of the sub-criticality assumption, and would apply in more general cases.
\\

To study the behavior of the population in the set of survival we introduce a new method based on comparing the stochastic dynamics with the deterministic one. In general we cannot expect that the stochastic dynamics is close to its mean. However, in the subcritical case we will prove that this holds true which explains why we  can describe simply the limit dynamics. 

 We use a method relying on control of variances based on maximum principles. A second  important step is to study the convergence of the exponents of the mean dynamics obtained by  the classical Hopf-Cole transform. These exponents satisfy  a discrete version of the classical PDE in \cite{BMP} for which convergence to the Hamilton-Jacobi equation is known.  An additional difficulty to show the convergence in our case comes from the discrete nature of the PDE which was studied in a more general setting in \cite{J}. Combining these two steps we prove the convergence in the limit of large population of the exponents to the viscosity solution of the corresponding Hamilton-Jacobi equation with cut-off.\\

The method used in this part is direct and do not rely on the classical tightness criteria, in contrast to \cite{CMMT} which makes use of Lipschitz estimates and martingale control. However, the convergence here is weaker: pointwise in time and locally uniformly in trait space. Whether a stronger convergence holds true remains unclear because,  first the exponents can takes the value $-\infty,$ and second our method does not allow to determine the behavior at the boundary between the extinction and survival sets.

\paragraph{Second part:} In this part, we study the super-critical branching regime considered in \cite{CMMT} in a more general trait space and under more realistic assumptions. In this setting all the initial subpopulation sizes are large, and we work in a uniformly super-critical regime, meaning that the subpopulation sizes remain large and the total population size is always infinite. This last fact makes the construction of the  model not obvious, because mutations can occur from an infinite number of subpopulations. Then, to justify the existence of the model, we introduce a new localization argument.\\
Our proof of the convergence toward the Hamilton-Jacobi equation follows a similar method as above. We first establish a uniform control on the normalized variances using the maximum principle that holds in super-critical regime. This allows us to deduce that the stochastic and deterministic dynamics are asymptotically equivalent, which implies that the stochastic and deterministic exponents are close.  We then conclude using the convergence of the Hopf-Cole transform of the mean dynamics to the associated Hamilton-Jacobi equation. Here we obtain a stronger convergence, locally uniform in time and space, and we deduce the tightness of the exponents $(\beta^K)_K$.\\

This paper is organized as follows: in Section 2, we define the model. Section 3 is devoted to the critical or subcritical regime. Finally, in Section 4, we study the supercritical regime. We refer to the beginning of Sections 3 and 4 for more details about each section. 
\section{ Model and motivation}
We recall the branching stochastic model introduced in \cite{CMMT} in a more general trait space $\R,$ instead of the one-dimensional torus in \cite{CMMT}. We consider a population parameterized by a scaling  parameter $K \in \mathbb{N}$, composed of individuals characterized by a trait which belongs to $\R,$ subject to  mutation. We consider a discretized trait space with discretization step  $\delta_K$ such that $\delta_K \rightarrow 0$ as $K\rightarrow +\infty$. The trait space associated to the scaling parameter $K$  is given by $\mathcal{X}_K := \{i\delta_K,~i \in \mathbb{Z}\}$. We define in a probabilistic  space $(\Omega,\mathcal{F},\mathbb{P})$ a Markov pure jump process  $(N^K_i(t), i\in\Z, t \geq 0)$ that models the size of the subpopulations, where for each 
$i\in \Z,$ the size of the subpopulation of trait $i\delta_K$ at time $t$ is $N^K_i(t).$ Our model is a  branching process in continuous time, so we only need to specify the individual birth and death rates. Any individual with trait $x\in \mathcal{X}_K:$
\begin{itemize}
    \item gives birth to a new individual with the same trait $x$  at rate $b(x);$
    \item dies at rate $d(x);$
    \item or gives birth to a new individual with  trait $y\in \mathcal{X}_K$ at rate 
    \begin{equation}
    \label{eq:proba of mutation}
        \delta_K p\log{K}G((y-x)\log{K}).
    \end{equation}
\end{itemize}
 Here, $b:\R\rightarrow \R^+$ is the birth rate,  $d:\R\rightarrow \R^+$ is the death rate and $p>0$ is the mutation rate. We assume that $G$ is a probability kernel describing mutation effects. It is scaled in \eqref{eq:proba of mutation} at logarithmic scale to make small mutations of sizes of the order $\frac{1}{\log{K}},$ and $\delta_K$ to make the Riemann sum corresponding to the global population mutation rate converge. In the remaining, we define $h_K=\delta_K\log{K}.$\\
 
 The process $(N^K_i(t),i\in \Z, t\geq 0)$ can be constructed using Poisson random  measures in a classical way (cf. e.g. \cite{FM,BM}). Let $(Q^b_i)_{i\in\Z},$ $(Q^d_i)_{i\in\Z},$ and $(Q^p_i)_{i\in\Z},$ be sequences of independents Poisson random measures on $\R^+\times \R^+$ with Lebesgue intensity, defined on the probabilistic space $(\Omega,\mathcal{F},\mathbb{P}).$ Informally speaking let us give some classical trajectorial representations about our process which will be justified in each setting in sections 3 and 4: for all $i\in \Z$
\begin{equation}
    \begin{aligned}
    \label{eq:PRM representation}
    N^K_i(t)=N^K_i(0)+\int_{0}^{t}\int_{\R^+}& \mathbf{1}_{y\leq b(i\delta_K)N^K_i(s^-)}Q^b_i(ds,dy)-\int_{0}^{t}\int_{\R^+} \mathbf{1}_{y\leq d(i\delta_K)N^K_i(s^-)}Q^d_i(ds,dy)\\
       &+ \sum_{j\in \mathbb{Z}} \int_{0}^{t}\int_{\R^+}\mathbf{1}_{y\leq ph_KG((j-i)h_K)N^K_j(s^-)}Q^p_j(ds,dy).
    \end{aligned}
\end{equation}
Furthermore, by compensating the Poisson random measures, we obtain the following  semi-martingale decomposition:
\begin{equation}
\label{eq:s-m}
    N^K_i(t)=N^K_i(0)+\int_{0}^{t}(b(i\delta_K)-d(i\delta_K))N^K_i(s)ds+\sum_{l\in \mathbb{Z}}ph_KG(lh_K)\int_{0}^{t}N^K_{l+i}(s)ds+M^K_i(t),
\end{equation}
where $M^K_i$ is a local martingale with quadratic variation (see \cite{BM})
\begin{equation}
    \langle M^K_i\rangle_t=\int_{0}^{t}(b(i\delta_K)+d(i\delta_K))N^K_i(s)ds+\sum_{l\in \mathbb{Z}}ph_KG(lh_K)\int_{0}^{t}N^K_{l+i}(s)ds.
\end{equation}
Note that this representation needs a justification, since $N^K_i$ appears both on the left and right hand side, and the stochastic integrals need to be well-defined. This is the case if for all $ T>0,$ and $i\in \Z,$  we have 
\begin{equation}
    \label{eq:hyp of existence}\mathbb{E}\Big(\sup_{t\in[0,T]}\big((b(i\delta_K)+d(i\delta_K))N^K_i(t)+\sum_{j\in\Z}ph_KG((j-i)h_K)N^K_j(t)\big)\Big)<\infty.
\end{equation}
Under this assumption we denote for all $i\in \mathbb{Z}$ and $t\geq 0,$ $n^K_i(t)=\Esp(N^K_i(t)),$ and we take the expectation in \eqref{eq:s-m}. We obtain the system of ordinary differential equations:
\begin{equation}
\label{eq:mean dynamics system}
\begin{cases}
    \frac{dn^K_i(t)}{dt}=(b(i\delta_K)-d(i\delta_K))n^K_i(t)+\sum_{l\in \mathbb{Z}}ph_KG(lh_K)n^K_{l+i}(t),~~(t,i)\in\mathbb{R}^+\times \mathbb{Z},\\
    n^K_i(0)=n^{K,0}(i\delta_K),
    \end{cases}
\end{equation}
where $n^{K,0}$ is the initial condition.\\

Our goal is to study the exponential growth of the population when $K\rightarrow +\infty$, and more precisely to describe the exponent of each subpopulation in the form $K^\beta$, where $\beta\in\{-\infty\}\cup[0,+\infty)$. To this end we introduce following \cite{CMMT} the Hopf-Cole transform at a logarithmic scale  
\begin{equation}
\beta^K_i(t)=\frac{\log{N^K_i(t\log{K})}}{\log{K}},
\end{equation}
and the corresponding transform of the mean dynamics
\begin{equation}
    u^K_i(t) = \frac{\log{n^K_i(t\log{K})}}{\log{K}},
    \end{equation}
    which satisfies the equation 
\begin{equation}
\label{eq:system of ODE}
\begin{cases}
  \frac{d}{dt}u^{K}_i(t)=b(i\delta_K)-d(i\delta_K)+\sum_{l\in\mathbb{Z}}ph_KG(lh_K)e^{\log{K}(u^{K}_{l+i}(t)-u^{K}_i(t))},~\forall(t,i)\in (0,+\infty)\times\mathbb{Z}, \\
  u^K_i(0)=u^{K,0}(i\delta_K).
  \end{cases}
\end{equation}
The above system is a discretized  version of the integro-differential equation studied in \cite{BMP}, where the solution converges to a viscosity solution of a Hamilton-Jacobi equation.  We refer to \cite{J} for the study of such discrete models in a more general setting.


Following \cite{CMMT,J}, we define the linear interpolation of $\beta^K_i,$ which allows us to define a dynamics in the continuous trait space $\R:$ for all $t>0$ and $x\in \mathbb{R}$ let $i\in \mathbb{Z}$ such that $x\in [i\delta_K,(i+1)\delta_K),$ \begin{equation}
\label{eq:linear interpolation}
    \widetilde{\beta}^{K}(t,x)=\beta^{K}_i(t)(1-\frac{x}{\delta_K}+i)+\beta^{K}_{i+1}(t)(\frac{x}{\delta_K}-i).
    \end{equation}
    The linear interpolation of $u^K_i$ is defined in the same way as above and  denoted by $\widetilde{u}^K.$\\
    The remaining of the paper is dedicated to the study of the convergence of $(\widetilde{\beta}^K)_K,$ as $K\rightarrow +\infty.$
    \subsection*{Assumptions A}
    Here, we assume the general classical assumptions from \cite{BMP} in a discrete setting close to \cite{J}.  
    \begin{enumerate}
    \item \label{item:1} We assume that $b$ and $d$  are Lipschitz-continuous functions, the mutation rate $p$ is constant, and  there exist positive constants $\overline{b}$ and $\overline{d}$ ,  such that for any $x \in \mathbb{R}$,
       \begin{equation}
       \label{eq: growth rates canditions}
0\leq b(x) \leq \overline{b},~~ 0\leq d(x) \leq \overline{d}.
       \end{equation}
       \item \label{item:2}
       The kernel $G$  has super-exponential  decay, i,e  $$\int_{\R}G(x)e^{ax}dx<+\infty,~~\forall a\in \R.$$
    \item \label{item:3} There exists a positive constant  $L$ such that,  for any $i\in \mathbb{Z}$  and $K\geq 1$
    \begin{align}
    \label{eq:Lipshitz conditon}
        \left\vert \frac{u^{K,0}((i+1)\delta_K)-u^{K,0}(i\delta_K)}{\delta_K}\right \vert \leq L,
        \end{align}
        i.e for any $K\geq 1,$ $u^{K,0}$ is $L$-Lipschitz continuous in the discrete trait space $\mathcal{X}_K$.
    \item \label{item:4} The linear interpolation $\widetilde{u}^{K,0}$ of $u^{K,0},$ converges locally uniformly to a continuous function $u^0,$ as $K\rightarrow +\infty$.
    \item Finally, we assume that for all $\varepsilon>0,$ we have 
    \begin{equation}
        \label{eq:discret cond}
       \frac{1}{K^\varepsilon} \ll \delta_K\ll \frac{1}{\log{K}}.
    \end{equation}
    \end{enumerate}
Assumption 2 is standard in the Hamilton–Jacobi approach (see \cite{BMP}) and ensures that the Hamilton–Jacobi equation is well-defined. Kernels with slow decay were studied in \cite{J}, leading to a non-standard Hamilton–Jacobi equations. Assumption 4 implies that $u^0$ is $L$-Lipschitz. Finally, Assumption 5 follows from \cite{CMMT} and implies in particular that $h_K\rightarrow 0,$ as $K\to +\infty.$.\\
In addition, for technical reason,  the mutation rate is assumed constant in this paper, this makes the spatial lipschitz bound of $u^K$ global.
\section{Critical or sub-critical case:}
In this section, we investigate the asymptotic behavior assuming a uniformly critical or subcritical dynamics. In Subsection 3.1, we state our assumptions. In Subsection 3.2, we establish preliminary results concerning well-posedness and the existence of second moments. In Subsection 3.3, we state our main result. Subsection 3.4 is devoted to the Hamilton–Jacobi approach applied to the mean dynamics, which constitutes the first step in the proof of the main result. In Subsection 3.5, we quantify the gap between the stochastic and deterministic dynamics, which is the key idea of the proof. Finally, Subsection 3.6 is devoted to the proof of the main result.  
\subsection{Assumptions B}
We state  our assumptions in the case of a uniformly critical or subcritical dynamics.
\begin{enumerate}
\item \label{item:6} There exist  positive constants $A_1$ and  $A_2$ such that for all $i\in \Z,$ and $K\geq 1$
    \begin{equation}
    \label{eq:sublinear growth initial conditon}
        u^{K,0}(i\delta_K)\leq -A_1\vert i\delta_K\vert +A_2.
    \end{equation}
    \item \label{item:66} We assume that 
\begin{equation}
    \label{eq:initial integrability}
    \mathbb{E}\Big( \big(\sum_{i\in\mathbb{Z}}N^K_i(0)\big)^2\Big)<+\infty.
\end{equation}
\item \label{item:7} We assume that there exists a positive constant $C$ such that
\begin{equation}
    \label{eq:uniform initial control1}
\sup_{i\in\mathbb{Z}}n^K_i(0)\mathbb{E}\Bigg (\Big( \frac{N^K_i(0)-n^K_i(0)}{n^K_i(0)}\Big)^2\Bigg)\leq C.
\end{equation}
\item \label{item:8} We assume a subcritical branching regime for all traits:
\begin{equation}
\label{eq:subrictical condition}
    \alpha:=\sup_{x\in \mathbb{R}}b(x)-d(x)+p\leq 0.
\end{equation}
\end{enumerate}
 Assumption B-\ref{item:6} is classical in the Hamilton-Jacobi approach (see for instance \cite{BMP,J}), it gives the affine decay of $u^K$. Moreover, it implies that $\sum_{i\in\Z}n^K_i(0)<+\infty.$ Therefore, by
  Borel-Cantelli lemma's we have for large enough $i,$ that $N^K_i(0)=0$ almost surely. Assumption B-\ref{item:66} guarantees the well-posedness of the dynamics \eqref{eq:PRM representation} and the integrability of the martingales involved in the semi martingale decomposition \eqref{eq:s-m}, as proved in Proposition 3.1. Assumption B-\ref{item:7} means that, initially the variance is bounded by the mean for each subpopulation. This is for example the case if $N^K_i(0)$ is a Poisson random variable with parameter $n^K_i(0).$ Note that $N^K_i(0)$ may be correlated random variables for different $i.$ Furthermore, combined with Assumption A-\ref{item:4}, this implies that $\beta^K(0)$ converges in probability to $u^0$ when $u^0>0$, and to $-\infty$ when $u^0<0.$  This will be a consequence of the arguments of section 3.6.

\subsection{Preliminary result}
\begin{pro}
Under Assumptions A-\ref{item:1}-\ref{item:2} and Assumption B-\ref{item:66}. The process $(N^K_i(t),i\in \Z,t\geq 0)$ is well-defined, and we have for all $T>0$
      \begin{equation}
      \label{eq:second order moment1}
      \mathbb{E}\bigg(\sup_{t\in[0,T]}\Big( \sum_{i\in \Z}N^K_i(t)\Big)^2\bigg)\leq C(T,K),
  \end{equation}
  where $C(T,K)$ is a positive constant depending on $T,K$ but not on $i.$
\end{pro}
\begin{proof}
  By standard coupling argument, we construct in $(\Omega,\mathcal{F},\mathbb{P})$ a Yule process $Z^K$ with parameter $\overline{b}+p$ and initial condition $\sum_{i\in \Z}N^K_i(0)$ such that
  \begin{equation}
  \label{eq:yule couplage}
      \sum_{i\in \Z}N^K_i(t)\leq Z^K_t,~~\text{almost surely}~~\forall t\geq 0.
  \end{equation}
Under assumption B-\ref{item:66}, the Yule process $Z^K$ is well-defined, and admits a finite second moment. Hence, \eqref{eq:hyp of existence} is verified, and thus $(N^K_i(t),i\in \Z,t\geq 0)$ is well-defined and satisfies \eqref{eq:second order moment1}.
\end{proof}
\subsection{Convergence result in the critical or subcritical dynamics}  
Let us state our main theorem.
\begin{thm}
Under Assumptions A-B, in the limit of $K\rightarrow +\infty,$  the stochastic process $(\widetilde{\beta}^K)_K\subset \mathbb{D}(\R^+,C(\R,\R))$ converges in the sense of finite-dimensional distributions in probability locally in space to the deterministic function 
\begin{equation}
\beta(t,x)=
    \begin{cases}
        u(t,x)~~~~\text{for}~~(t,x)\in \{(t,x)\in [0,+\infty)\times \mathbb{R},~~ u(t,x)>0\},\\
        -\infty~~~~~~\text{for}~~(t,x)\in \{(t,x)\in [0,+\infty)\times \mathbb{R},~~ u(t,x)<0\}, 
    \end{cases}
\end{equation}  
where $u$ is the unique Lipschitz continuous viscosity solution of the Hamilton-Jacobi equation: 

 \begin{equation}
  \label{eq:HJ}
  \begin{cases}
      \partial_t u(t,x)=b(x)-d(x)+p\int_{\mathbb{R}}G(y)e^{\nabla u(t,x).y} dy,~~\forall (t,x)\in (0,+\infty)\times \R,\\
      u(0,.)=u^0(.).
       \end{cases}
  \end{equation}
More precisely, we prove that for all $\eta>0,$ $t>0$ and real numbers $a<b,$\\
if $\{t\}\times [a,b]\subset \{u>0\},$ we have
\begin{equation}
\label{eq:proba convergence}
    \lim_{K\rightarrow +\infty} \Prob(\sup_{x\in[a,b]}\vert \widetilde{\beta}^K(t,x)-u(t,x)\vert >\eta)=0,
\end{equation}
 if  $\{t\}\times[a,b]\subset \{u<0\},$ we have
\begin{equation}
\label{eq:convergence to infinity}
    \lim_{K\rightarrow +\infty} \Prob(\sup_{x\in[a,b]} \widetilde{\beta}^K(t,x)=-\infty)=1.
    \end{equation}
    \end{thm}

    This theorem provides a probabilistic derivation of a  Hamilton-Jacobi equation with cut-off. It characterizes the space-time domain of survival and describes within this set the limit behavior of exponents of the population using a Hamilton-Jacobi approach. In addition this result shows that the exponents go to $-\infty$ in the set where the viscosity solution is negative. In other words \eqref{eq:convergence to infinity} means that for large $K$ the subpopulations are extinct.\\ 

The proof of this result is divided into two main steps. The first step consists in proving that the linear interpolation of the  Hopf-Cole transform of the mean dynamics $\widetilde{u}^K,$ converges to the viscosity solution $u$ of the associated Hamilton-Jacobi equation. This is based on the Arzela-Ascoli theorem, using local uniform bounds and local uniform Lipschitz estimates in both time and space as in \cite{BMP} but in discretized setting. The second step, the key step of the proof. It consists in comparing the stochastic dynamics with the deterministic one. This is achieved by establishing a uniform control of variances over all subpopulations, which allows us to prove that, in probability $\beta^K$ and $u^K$ remain close on the set $\{u>0\},$ when the parameter $K$ is large.
As a consequence of the first step, we conclude that the stochastic exponent $\beta^K$ converges to the viscosity solution of the corresponding Hamilton-Jacobi equation on the set where this solution is positive. On the set where the viscosity solution is negative, the first step implies that the mean dynamics $n^K$ converges to zero. The integer-valued nature of the subpopulation sizes then ensures that the stochastic dynamics also converges to zero. Finally, we note that the convergence of the exponent to $-\infty$ in $\{u<0\}$ holds true in general and in particular does not require any sub-criticality or super-criticality assumption.

\subsection{Hamilton-Jacobi approach for the mean dynamics}
In this subsection, we investigate the Hamilton-Jacobi approach for the mean dynamics as a first step to analyze the stochastic dynamics.
\begin{thm}
  Under Assumptions A and B-\ref{item:6}, in the limit $K\rightarrow +\infty,$ the sequence of functions $(\widetilde{u}^K)_K$  converges locally uniformly in time and space to the unique viscosity solution of \eqref{eq:HJ}.
\end{thm}
\begin{proof}
  To prove this result, we use the Arzelà–Ascoli theorem. Let $T>0.$ From Assumptions A-\ref{item:1}-\ref{item:3}, and B-\ref{item:6}, Lemma 4.1 and Proposition 4.2 in \cite{J}, we have for all $(t,i)\in [0,T]\times \Z$ that
  \begin{equation}
  \label{eq:local bounds}
 -L\vert i\delta_K\vert+A_4+A_5t\leq u^K_i(t)\leq -A_1\vert i\delta_K\vert +A_2+A_3t,
     \end{equation}
  and 
  \begin{equation}
  \label{eq: spatial lischitz estimates}
      \vert u^K_{i+1}(t)-u^K_i(t)\vert\leq \frac{\Vert p\Vert_{Lip}}{p}\vert u^K_{i+1}(t)\vert\delta_K+ C_{Lip}(T)\delta_K=C_{Lip}(T)\delta_K,
      \end{equation}
where $A_3,A_4,A_5,A_6$ are constants independent of $K$, and $C_{Lip}(T)$ is a positive constant independent of $K,$ but depends on $T.$
We deduce a local bound, and  spatial Lipschitz estimates. Now, we prove the Lipschitz estimates in time. For $K$ large enough, we have 
\begin{align*}
  \vert \partial_t u^K_i(t)\vert &\leq \overline{b} +\overline{d}+p\sum_{l\in\mathbb{Z}}h_KG(lh_K)e^{C_{Lip}(T)\vert l h_K\vert }\\
  &\leq \overline{b} +\overline{d}+2p \int_{\mathbb{R}}G(x)e^{C_{Lip}(T)\vert x\vert}dx.
\end{align*}
Thus, we obtain uniform local Lipschitz estimates in time. Then, by the Arzelà-Ascoli theorem, we deduce the convergence along a subsequence of $K$ to a continuous function $u$, and by classical arguments of viscosity solutions and Assumption A-\ref{item:4}, we conclude that $u$ is a viscosity solution of \eqref{eq:HJ}. Moreover, by  uniqueness of viscosity solutions of \eqref{eq:HJ}, we deduce the convergence of the full sequence $(\widetilde{u}^K)_K$ to $u.$ 
\end{proof}
\subsection{The gap between the stochastic and the deterministic dynamics }
In this subsection, we state and prove the next result which compares the stochastic and the deterministic dynamics. 
\begin{pro}
  Let $T>0.$ We have for all $(t,i)\in[0,T]\times \Z,$
    \begin{equation}
\label{eq:controle}n^K_i(t\log{K})\mathbb{E}\Bigg(\Big(\frac{N^K_i(t\log{K})-n^K_i(t\log{K})}{n^K_i(t\log{K})}\Big)^2\Bigg)\leq C(K^{(\alpha+o(1))t}+C(T)\log{K}),
\end{equation}
where $C$ is a positive constant independent of $K,$ $C(T)$ is a positive constant independent of $K,$ and depending on $T,$ and $o(1)$ a constant depending only on $K,$ which goes to zero when $K\rightarrow +\infty.$ Recall that $\alpha$ has been defined in \eqref{eq:subrictical condition}. Moreover, at a logarithmic time scale $N^K/n^K$ converges in $L^2$ to $1$ in $\{u>0\},$ i,e, for all $t>0$ and real numbers $a<b,$ such that $\{t\}\times [a,b]\subset \{u>0\},$ we have
\begin{equation}
    \mathbb{E}\Big( \sup_{\{i\in \Z, i\delta_K\in[a,b]\}}\Big(\frac{N^K_i(t\log{K})}{n^K_i(t\log{K})}-1\Big)^2\Big)\rightarrow 0,~~\text{as}~~K\rightarrow+\infty.
    \end{equation}
\end{pro}
This result is a consequence of the branching property, and implies that the normalized variances are small when the mean dynamics is large. The proof relies on Itô's lemma
 and the maximum principle. 
  This is an original proof which, to the best of our knowledge, is non standard in  the theory of branching processes. 
\begin{proof}
   Let $T>0.$ Let us define for all $(t,i)\in [0,T]\times \mathbb{Z}
    ,$ $S^K_i(t)=\frac{(N^K_i(t)-n^K_i(t))^2}{n^K_i(t)}.$ Using the semi martingale decomposition in \eqref{eq:s-m} which is well-defined by \eqref{eq:second order moment1}, and by Ito's lemma, we obtain for all $t>s>0,$
\begin{align*}
S^K_i(t)=&S^K_i(s)+\int_{s}^{t}(b(i\delta_K)-d(i\delta_K))S^K_i(\tau)d\tau+2\int_{s}^{t}p\sum_{l\in \mathbb{Z}}h_KG(lh_K)\\
    &~~~~~~~~~~~~~~~~~~~~~~~~~~~~~~~~~~~~~~~~~~~~~~~~~~~~~~~~~~~~~~~~~~~~~~~~~~\frac{(N^K_i(\tau)-n^K_i(\tau))(N^K_{l+i}(\tau)-n^K_{l+i}(\tau))}{n^K_i(\tau)}d\tau\\
    &-\int_{s}^{t}p\sum_{l\in \mathbb{Z}}h_KG(lh_K)\frac{(N^K_i(\tau)-n^K_i(\tau))^2n^K_{l+i}(\tau)}{(n^K_i(\tau))^2}d\tau+\int_{s}^{t}\frac{d\langle M^K_i\rangle_\tau}{n^K_i(\tau)}+\widetilde{M}^K_i(t)-\widetilde{M}^K_i(s),
\end{align*}
where $ t\mapsto \widetilde{M}^K_i(t)$ is a martingale, since all quantities in the above decomposition are bounded at most by second moments of the population sizes, and by \eqref{eq:second order moment1} we have 
$$\mathbb{E}\Big(\sup_{i\in \Z}\sup_{t\in[0,T]}\big(N^K_i(t)\big)^2\Big)\leq C(T,K).$$ We define $Y^K_i(t)=\mathbb{E}(S^K_i(t)),$ which is well-defined, and we take the expectation in the above equation, and by Young's inequality we  deduce that
\begin{align*}
    Y^K_i(t)\leq& Y^K_i(s)+\int_{s}^{t}\big(b(i\delta_K)-d(i\delta_K)+p\sum_{l\in \mathbb{Z}}h_KG(lh_K)\big)Y^K_i(\tau))d\tau   \\
    &+\int_{s}^{t}p\sum_{l\in \mathbb{Z}}h_KG(lh_K)\frac{n^K_{l+i}(\tau)}{n^K_i(\tau)}( Y^K_{l+i}(\tau))- Y^K_i(\tau))d\tau+\int_{s}^{t}\frac{d  \mathbb{E}(\langle M^K_i\rangle_\tau)}{n^K_i(\tau)}.
\end{align*}
Moreover, as mentioned in the proof of Theorem 3.3, since the mutation rate is assumed to be constant, the Lipchitz bound in space \eqref{eq: spatial lischitz estimates} is global, and then  
$$\sum_{l\in\mathbb{Z}}h_KG(lh_K)\frac{n^K_{l+i}(\tau)}{n^K_i(\tau)}\leq \sum_{l\in\mathbb{Z}}h_KG(lh_K)e^{C_{Lip}(T/\log{K})\vert lh_K\vert}.$$ 
By Assumptions A-\ref{item:1} and A-\ref{item:2}, we deduce that
\begin{equation}
     \frac{1}{n^K_i(\tau)}\frac{d\mathbb{E}(\langle M^K_i\rangle_\tau)}{d\tau}=b(i\delta_K)+d(i\delta_K)+p\sum_{l\in\mathbb{Z}}h_KG(lh_K)\frac{n^K_{l+i}(\tau)}{n^K_i(\tau)}\leq C(T/\log{K}),
\end{equation}
where $C(.)$ is a positive function independent of $K,t$ and $i.$
Moreover, by Assumption  B-\ref{item:8}, we have \begin{equation}
   b(i\delta_K)-d(i\delta_K)+p\sum_{l\in \mathbb{Z}}h_KG(lh_K))\leq \alpha+o(1),
\end{equation}
where $o(1)$ is a negligible constant depending only on $K$ and going to zero when $K$ goes to infinity. Therefore  
\begin{equation}
\label{eq:maximum principle form1}
    \frac{d}{dt}Y^K_i(t)\leq (\alpha+o(1))Y^K_i(t)+ p\sum_{l\in \mathbb{Z}}h_KG(lh_K)\frac{n^K_{l+i}(t)}{n^K_i(t)}(Y^K_{l+i}(t)-Y^K_i(t))+C(T/\log{K}).
\end{equation}
By the maximum principle (see Appendix A), we conclude that
\begin{equation}
    Y^K_i(t)\leq e^{(\alpha+o(1))t}\sup_{i\in\mathbb{Z}}Y^K_i(0)+C(T/\log{K})t.
\end{equation}
Hence \eqref{eq:controle} is proved.\\
Let us now prove the last point of the Proposition 3.4. Let $t>0$ and real numbers $a<b,$ such that $\{t\}\times [a,b]\subset \{u>0\},$ we have by the convergence of $\widetilde{u}^K$ to $u$ that for all $i\delta_K\in[a,b]$
\begin{equation*}
    n^K_i(t\log{K})=K^{u^K_i(t)}\geq K^{\inf_{x\in[a,b]}u(t,x)/2}.
\end{equation*}
Therefore, by \eqref{eq:controle} we obtain that
\begin{align*}
    \mathbb{E}\Big( \sup_{\{i\in \Z, i\delta_K\in[a,b]\}}\Big(\frac{N^K_i(t\log{K})}{n^K_i(t\log{K})}-1\Big)^2\Big)&\leq \sum_{i\in \Z,~ i\delta_K\in[a,b]} CK^{-\inf_{x\in[a,b]}u(t,x)/2} (K^{(\alpha+o(1))t)}+C(t)\log{K})\\
    &\leq \frac{K^{-\inf_{x\in[a,b]}u(t,x)/4}}{\delta_K}.
\end{align*}
From \eqref{eq:discret cond}, we conclude the $L^2$ convergence. 
\end{proof}
\subsection{Proof of Theorem 3.2}
\subsubsection*{Convergence in $\{u<0\}$:} Let $t\geq 0$ and  $a\leq b$ real numbers such that $\{t\}\times [a,b]\subset \{u<0\}.$  We have  by the uniform convergence of $\widetilde{u}^K$ to $u$ and the fact that $\sup_{x\in[a,b]}u(t,x)<0$ that
\begin{equation}
\label{eq:vitesse mean convergence}
n^K_i(t\log{K})=K^{u^K_i(t)}\leq K^{\sup\limits_{\{x\in [a,b]\}}u(t,x)/2}\rightarrow 0,~~\text{as}~~K\rightarrow +\infty,~~~\forall i\in \Z, i\delta_K\in[a,b].
\end{equation}
Thus 
\begin{align*}
    \mathbb{P}(\sup_{\{i\in \mathbb{Z}/ i\delta_K\in [a,b]\}}N^K_i(t\log{K})\geq 1)&\leq \Esp\big(\sup_{\{i\in \mathbb{Z}/ i\delta_K\in [a,b]\}}N^K_i(t\log{K})\big)\\
    &\leq \sum_{\{i\in \mathbb{Z}/ i\delta_K\in [a,b]\}}n^K_i(t\log{K})\\
    &\leq C \frac{K^{\sup\limits_{\{x\in [a,b]\}}u(t,x)/2}}{\delta_K}.
\end{align*}
By \eqref{eq:discret cond} this probability goes to zero when $K\rightarrow +\infty.$  Moreover, the supremum of $N^K_i$ is integer-valued, hence
\begin{equation}
    \sup_{\{i\in \mathbb{Z}/ i\delta_K\in [a,b]\}}N^K_i(t\log{K}))=0,
\end{equation}
with probability converging to 1 when $K\rightarrow +\infty.$ Furthermore, $\{u<0\}$ is an open set we can deduce using similar arguments that $\max(\widetilde{\beta}^K(t,a),\widetilde{\beta}^K(t,b))$ converges to $-\infty$ with probability converging to 1. Hence, we deduce \eqref{eq:convergence to infinity}.
\subsubsection*{Convergence in $\{u>0\}$.} 
Let $\eta>0,$  $t\geq 0$ and  $a\leq b$ real numbers such that $\{t\}\times [a,b]\subset \{u>0\}.$ Let $i\in \Z$  such that $i\delta_K\in[a,b],$ we have by Proposition 3.4 and Theorem 3.3 for large enough $K$ that
\begin{align*}
    \mathbb{P}\left( \left| \beta^K_i(t)-u^K_i(t) \right| > \eta \right)& = \mathbb{P}\Bigg(\frac{1}{\log{K}} \left| \log\left(1 + \frac{N^K_i(t\log{K}) - n^K_i(t\log{K})}{n^K_i(t\log{K})} \right) \right| > \eta \Bigg)\\
    &\leq \mathbb{P}\Bigg( \frac{1}{\log{K}} \left|  \frac{N^K_i(t\log{K}) - n^K_i(t\log{K})}{n^K_i(t\log{K})}  \right| > \eta \Bigg)+\mathbb{P}\Bigg(  \left|  \frac{N^K_i(t\log{K}) - n^K_i(t\log{K})}{n^K_i(t\log{K})}  \right| > 1/2 \Bigg)\\
    &\leq \Big(\frac{1}{\eta^2\log^2{K}}+4\Big)\mathbb{E}\Bigg( \frac{N^K_i(t\log{K}) - n^K_i(t\log{K})}{n^K_i(t\log{K})}  \Bigg)^2\\
    & \leq C\Big(\frac{1}{\eta^2\log^2{K}}+4\Big)K^{-u^K_i(t)}(K^{(\alpha+o(1))t}+C(t)\log{K})\\
    &\leq C\Big(\frac{1}{\eta^2\log^2{K}}+4\Big)K^{- {\inf\limits_{\{x\in [a,b]\}}}u(t,x)/2}(K^{(\alpha+o(1))t}+C(t)\log{K})\\
    &\leq C K^{-{\inf\limits_{\{x\in [a,b]\}}}u(t,x)/4}.
    \end{align*}
   In the same way we can prove that $\max( \left| \widetilde{\beta}^K(t,a)-\widetilde{u}^K(t,a) \right|,\left| \widetilde{\beta}^K(t,b)-\widetilde{u}^K(t,b) \right|)$ converges in probability to zero. Thus,from \eqref{eq:discret cond}, we deduce that
    \begin{equation}
     \mathbb{P}\left( \sup\limits_{\{x\in [a,b]\}} \left| \widetilde{\beta}^K(t,x)- \widetilde{u}^K(t,x)\right| > \eta \right) \rightarrow 0,~~\text{as} ~~K\rightarrow +\infty. 
    \end{equation}
   Therefore, by Theorem 3.3, we conclude the proof of \eqref{eq:proba convergence}.
\section{Super-critical case}
This section is devoted to the asymptotic analysis of the supercritical dynamics i,e the birth rate is greater than or equal to the death rate.
In Subsection 4.1, we state our assumptions.
Subsection 4.2 is devoted to preliminary results concerning the existence of the stochastic model and the Hamilton–Jacobi approach for the mean dynamics.
In Subsection 4.3, we state our main result and provide a result analogous to Proposition 3.4, which is the key  in the proof of the main theorem. Finally, Subsection 4.4 contains the proof of the convergence toward the viscosity solution.
\subsection{Assumptions C}
\begin{enumerate}
    \item \label{item:9} We assume that there exist two positive constants $A,B$ such that for all $i\in \Z$
\begin{equation}
\label{eq:initial linear growth2}
    u^K_i(0)\leq A\vert i\delta_K\vert +B.
\end{equation}
\item \label{item:10} We assume that there exists $a>0,$ such that for all $i\in \Z,$ we have 
\begin{equation}
\label{eq:large initial mean dynamics}
    n^K_i(0)\geq K^a.
 \end{equation}
\item \label{item:11} We assume a uniformly supercritical dynamics, i,e, for all $x\in \R$ \begin{equation}
 \label{eq:super-critical case}
     b(x)\geq d(x).
 \end{equation}

\item \label{item:12}  Finally, we assume a similar condition as \eqref{eq:uniform initial control1}, i.e, there exists a positive constant $C$ such that 
 \begin{equation}
\label{eq:initial control supcrtitical case}
\sup_{i\in\mathbb{Z}}\mathbb{E}\Bigg (\Big( \frac{N^K_i(0)-n^K_i(0)}{n^K_i(0)}\Big)^2\Bigg)\leq \frac{C}{K^a}.
 \end{equation}
 \end{enumerate}
 Assumption C-\ref{item:9} is new in the framework of Hamilton-Jacobi approach. The classical assumption \eqref{eq:sublinear growth initial conditon} is stronger and allows to perform the analysis in the classical functional spaces. Here, we will introduce a new functional space depending on apriori bounds on the solution of the Hamilton-Jacobi equation. Assumption C-\ref{item:10} is a strong assumption meaning that initially the density of all populations is large, but this assumption is weaker than that one in \cite{CMMT} where a minimal initial population was assumed for the stochastic dynamics $N^K(0)$. In \eqref{eq:super-critical case},  we include the possibility that the birth rate and the death rate can be equal which includes the setting of \cite{CMMT}. Furthermore, combining \eqref{eq:large initial mean dynamics} with \eqref{eq:super-critical case} we obtain that for all $t>0$ and $i\in \Z$ 
 \begin{equation}
 \label{eq:large dynamics}
     n^K_i(t)\geq K^a.
 \end{equation}
In this setting the initial total population size can be infinite, i.e, $\sum_{i\in \Z}N^K_i(0)=+\infty$ almost surely. This is clear in the case where the subpopulation sizes $N^K_i$ are independent. Let us prove that. By Paley–Zygmund inequality and from C-\ref{item:10}, we have for all $i\in \Z$ that
 \begin{align*}
     \mathbb{P}(N^K_i(0)>K^{a/2})=& \mathbb{P}\big(N^K_i(0)>\frac{K^{a/2}}{n^K_i(0)}n^K_i(0)\big)\\
    &\geq \big(1- \frac{K^{a/2}}{n^K_i(0)}\big)^2\frac{n^K_i(0)^2}{\mathbb{E}(N^K_i(0)^2)}\\
    &\geq  (1- K^{-a/2})^2\frac{n^K_i(0)^2}{\mathbb{E}(N^K_i(0)^2)}.
 \end{align*}
Moreover, from C-\ref{item:12} we have $\mathbb{E}(N^K_i(0)^2)/(n^K_i(0)^2)\leq C(1+\frac{1}{K^a}).$ Then for $K$ large enough we have 
 \begin{equation}
     \mathbb{P}(\forall i\in \Z,~N^K_i(0)>K^{a/2})=1.
     \end{equation}
     Thus, $\sum_{i\in \Z}N^K_i(0)=+\infty$almost surely.
\subsection{Preliminary results and Hamilton-Jacobi approach for the mean dynamics}
The next proposition shows the existence of the model, and provides the integrability of the first and second moments involved in the mutation term. Note that due the fact that the initial population maybe infinite there is no first jump time in \eqref{eq:PRM representation}. Hence the notion of solution should be carefully defined. We will say that a process $(N^K_i(t),i\in\Z, t\geq 0)$ defined on the probability space $(\Omega,\mathcal{F},\mathbb{P})$  is solution to our branching population dynamics if for all $i\in \Z,$ the process $(N^K_i(t),t\geq 0)$ is almost surely finite and satisfies almost surely \eqref{eq:PRM representation}.
\begin{pro}
    Under Assumptions A-C, there exist a process $(N^K_i(t),i\in \Z,t\geq 0)$ solution to our branching dynamics, and for all $T>0,$ we have 
    \begin{equation}
    \label{eq:moments control}
        \max\Bigg( \mathbb{E}\Big(\sup_{t\in[0,T]} \sum_{i\in \mathbb{Z}}e^{-C_A\vert ih_K\vert}N^K_i(t)\Big),\mathbb{E}\Big(\sup_{t\in[0,T]} \sum_{i\in \mathbb{Z}}e^{-C_A\vert ih_K\vert}(N^K_i(t))^2\Big)\Bigg)<+\infty,
    \end{equation}
    where $C_A:=2A+1.$
\end{pro}

\begin{proof}
  To prove that the process $N^K$ is well-defined, we introduce new arguments. First, we approximate the process by a well-defined process that implies the existence of the process as a mathematical object, then we prove that it is finite by introducing a judicious stopping time and applying localization arguments.\\
  
  This model was studied in \cite{CMMT} in the torus, in which  there is a finite number of  finite subpopulations which makes the existence standard. But here, the initial total population size is infinite. Then, first we justify the existence of the process $N^K$ and hence we prove that the size of each subpopulation is finite and satisfies \eqref{eq:moments control}. Let us justify the existence of a process $N^K$ solution to \eqref{eq:PRM representation}. To this end let $M>0.$ We define on the same probability space the stochastic process $N^{K,M}$ defined as follows for all $t>0,$ $i\in \Z,~\vert i\delta_K\vert<M$
\begin{equation}
\label{eq:approximaated process}
\begin{aligned}
N^{K,M}_i(t)=N^K_i(0)+\int_{0}^{t}\int_{\R^+}& \mathbf{1}_{y\leq b(i\delta_K)N^{K,M}_i(s^-)}Q^b_i(ds,dy)-\int_{0}^{t}\int_{\R^+} \mathbf{1}_{y\leq d(i\delta_K)N^{K,M}_i(s^-)}Q^d_i(ds,dy)\\
       &+ \sum_{j\in \mathbb{Z},~\vert j\delta_K\vert 
       <M} \int_{0}^{t}\int_{\R^+}\mathbf{1}_{y\leq ph_KG((j-i)h_K)N^{K,M}_j(s^-)}Q^p_j(ds,dy),
\end{aligned}
\end{equation}
 and $N^{K,M}_i(t)=0,$ for $i\in \Z,~\vert i\delta_K\vert\geq M.$ 
Note that for all $M>0,$ the process $N^{K,M}$ is well-defined because the initial population is finite and with finite second order moment and there is at most a finite number of jumps in the stochastic integrals above. This can be proved under \eqref{eq:initial linear growth2} by classical coupling arguments with a Yule Process as in the proof of Proposition 3.1. Moreover, $N^{K,M}_i(t)$ is non-decreasing in $M,$ so the limit when $M$ goes to infinity exists almost surely for all $i\in \Z$ and $t\in \R^+$. Hence the existence of $N^K,$ as increasing limit of $N^{K,M}.$ Moreover letting $M$ go to infinity in \eqref{eq:approximaated process} we deduce that $N^K$ is a solution to \eqref{eq:PRM representation}.\\
Let us now show that $N^K$ is finite and satisfies \eqref{eq:moments control}.
 We introduce the stopping time $\tau_M,$ defined for all $\theta>0,$ by 
   \begin{equation}
       \tau_\theta=\inf\{t>0,~\sum_{i\in \mathbb{Z}}e^{-C_A\vert ih_K\vert}N^K_i(t)\geq \theta\}.
   \end{equation}
   Let us prove that for any $T>0,$ we have 
   \begin{equation}
   \label{eq:stoping time}
       \mathbb{P}(\sup_{\theta\geq 1}\tau_\theta\geq T)= 1.
   \end{equation}
   Let $T>0.$ For all $t\in [0,\tau_\theta\wedge T],$ we have 
   \begin{align*}
       N^K_i(t\wedge \tau_\theta)\leq  N^K_i(0)+&\int_{0}^{t\wedge \tau_\theta}\int_{\mathbb{N}}\mathbf{1}_{y\leq b(i\delta_K)N^K_i(s^-)}Q^b_i(ds,dy)\\
       &+ \sum_{j\in \mathbb{Z}} \int_{0}^{t\wedge \tau_\theta}\int_{\mathbb{R}^+}\mathbf{1}_{y\leq ph_KG((j-i)h_K)N^K_j(s^-)}Q^p_j(ds,dy).
 \end{align*}
 Then, 
\begin{equation}
\label{eq: stopping time ineq}
\begin{aligned}
     \sup_{t\in [0,T\wedge \tau_\theta]}\sum_{i\in \mathbb{Z}}e^{-C_A\vert ih_K\vert} N^K_i(t)&\leq  \sum_{i\in \mathbb{Z}}e^{-C_A\vert ih_K\vert} N^K_i(0)+\sum_{i\in \mathbb{Z}}e^{-C_A\vert ih_K\vert} \int_{0}^{T\wedge \tau_\theta}\int_{\mathbb{R}^+}\mathbf{1}_{y\leq b(i\delta_K)N^K_i(s^-)}Q^b_i(ds,dy)\\
     &+\sum_{i\in\mathbb{Z}} \sum_{j\in \mathbb{Z}}e^{-C_A\vert ih_K\vert} \int_{0}^{T\wedge \tau_\theta}\int_{\mathbb{R}^+}\mathbf{1}_{y\leq ph_KG((j-i)h_K)N^K_j(s^-)}Q^p_j(ds,dy).
 \end{aligned}
 \end{equation}
 Furthermore for any $(t,i)\in \R^+\times \Z,$ we have 
 $$N^K_i(t\wedge \tau_\theta)\leq e^{C_A\vert ih_K\vert}\theta,$$
and
\begin{align*}
\sum_{j\in \Z}h_KG((j-i)h_K)N^K_j(t\wedge \tau_\theta)&= \sum_{j\in \Z}h_Ke^{C_A\vert jh_K\vert}G((j-i)h_K)e^{-C_A\vert jh_K\vert}N^K_j(t\wedge \tau_\theta)\\
&\leq  e^{C_A\vert ih_K\vert}\theta\sum_{j\in\Z}h_Ke^{C_A\vert lh_K\vert}G(lh_K).
\end{align*}
 Hence, the local martingale involved in the decomposition of $N^K_i$, is a martingale before $T\wedge\tau_\theta.$ Therefore, we take the expectation in \eqref{eq: stopping time ineq}, to obtain 
 \begin{align*}
     \mathbb{E}\big(\sup_{t\in [0,T\wedge \tau_\theta]}\sum_{i\in \mathbb{Z}}e^{-C_A\vert ih_K\vert} N^K_i(t) \big)&\leq \mathbb{E}\big(\sum_{i\in \mathbb{N}}e^{-C_A\vert ih_K\vert} N^K_i(0) \big)+\mathbb{E}\Big(\int_{0}^{T\wedge \tau_\theta} \sum_{i\in \mathbb{Z}}e^{-C_A\vert ih_K\vert} b(i\delta_K)N^K_i(s)\Big)\\
     &~~+\mathbb{E}\Big(\int_{0}^{T\wedge \tau_\theta}\sum_{i\in\mathbb{Z}} \sum_{j\in \mathbb{Z}} ph_KG((j-i)h_K)e^{-C_A\vert ih_K\vert }N^K_j(s)ds\Big)\\
     &\leq  \mathbb{E}\big(\sum_{i\in \mathbb{N}}e^{-C_A\vert ih_K\vert} N^K_i(0) \big) +\overline{b}  \mathbb{E}\Big( \int_{0}^{T\wedge \tau_\theta}\sum_{i\in \mathbb{Z}}e^{-C_A\vert ih_K\vert}  N^K_i(s)ds\Big)\\
&~~+p\mathbb{E}\Big(\int_{0}^{T\wedge \tau_\theta}\sum_{j\in\mathbb{Z}} \sum_{i\in \mathbb{Z}} h_Ke^{C_A\vert (j-i)h_K\vert}G((j-i)h_K)e^{-C_A\vert jh_K\vert }N^K_j(s)ds\Big) \\
     &\leq \mathbb{E}\big(\sum_{i\in \mathbb{Z}}e^{-C_A\vert ih_K\vert} N^K_i(0) \big)+C \int_{0}^{T} \mathbb{E}\Big( \sum_{i\in \mathbb{Z}}e^{-C_A\vert ih_K\vert}  N^K_i(s\wedge \tau_\theta)ds\Big). 
 \end{align*}
 Note that $T\mapsto \mathbb{E}\big(\sup_{t\in [0,T\wedge \tau_\theta]}\sum_{i\in \mathbb{Z}}e^{-C_A\vert ih_K\vert} N^K_i(t) \big) $ is measurable and bounded by $\theta.$ Thus,
 by Gronwall's inequality and \eqref{eq:initial linear growth2}, we obtain that for any $\theta>0$
 \begin{equation}
 \label{eq:gronwal form}
       \mathbb{E}\big(\sup_{t\in [0,T\wedge \tau_\theta]}\sum_{i\in \mathbb{Z}}e^{-C_A\vert ih_K\vert} N^K_i(t) \big)\leq e^{CT}\mathbb{E}\big(\sum_{i\in \mathbb{Z}}e^{-C_A\vert ih_K\vert} N^K_i(0) \big)\leq C(T,K).
 \end{equation}
By contradiction we assume that \eqref{eq:stoping time} fails, then there exists $T>0$ such that $ \mathbb{P}(\sup_{\theta\geq 1}\tau_\theta< T):=a_T>0.$ Therefore  $\mathbb{E}\big(\sup_{t\in [0,T\wedge \tau_\theta]}\sum_{i\in \mathbb{N}}e^{-C_A\vert ih_K\vert} N^K_i(t) \big)>a_T\theta\rightarrow  +\infty, \text{as}~~\theta\rightarrow +\infty.$ 
This is in contradiction with \eqref{eq:gronwal form}. Finally, by Fatou's Lemma and \eqref{eq:gronwal form}, we deduce that 
\begin{equation*}
    \mathbb{E}\big(\sup_{t\in [0,T]}\sum_{i\in \mathbb{Z}}e^{-C_A\vert ih_K\vert} N^K_i(t) \big)\leq C(T,K).
\end{equation*}
The proof of the second moment follows similar arguments. Let us give the main points. We introduce a new stopping time as   \begin{equation*}
       \tau'_\theta=\inf\{t>0,~\sum_{i\in \mathbb{Z}}e^{-C_A\vert ih_K\vert}(N^K_i(t))^2\geq \theta\}.
   \end{equation*}
The representation of $(N^K)^2$ is given as follows
\begin{equation}
\label{eq:eqqq}
    \begin{aligned}
      \big( N^K_i(t)\big)^2=&\big(N^K_i(0)\big)^2+\int_{0}^{t}\int_{\R^+} (2N^K_i(s^-)+1)\mathbf{1}_{y\leq b(i\delta_K)N^K_i(s^-)}Q^b_i(ds,dy)-\int_{0}^{t}\int_{\R^+}(2N^K_i(s^-)-1) \\
      &\mathbf{1}_{y\leq d(i\delta_K)N^K_i(s^-)}Q^d_i(ds,dy)
       + \sum_{j\in \mathbb{Z}} \int_{0}^{t}\int_{\mathbb{R}^+}(2N^K_i(s^-)+1)\mathbf{1}_{y\leq ph_KG((j-i)h_K)N^K_j(s^-)}Q^p_j(ds,dy).
    \end{aligned}
\end{equation}
The only point which differs from the previous proof is the fact that the local martingale appearing in the mutation term in \eqref{eq:eqqq} is a martingale before $\tau'_\theta.$ 
For all $(t,i)\in (0,+\infty)\times \Z,$ we have 
\begin{align*}
    \sum_{j\in \mathbb{Z}} N^K_i(t\wedge \tau'_\theta) h_KG((j-i)h_K)N^K_j(t\wedge \tau'_\theta)&\leq \sum_{j\in \mathbb{Z}} \frac{1}{2}((N^K_i(t\wedge \tau'_\theta))^2+(N^K_j(t\wedge \tau'_\theta))^2)h_KG((j-i)h_K) \\
    &\leq \frac{1}{2}\theta e^{C_A\vert ih_K\vert}\sum_{j\in \Z}h_K(1+e^{-\vert jh_K\vert})G(jh_K).
\end{align*}
Hence, we can take the expectation in \eqref{eq:eqqq}, the remainder of the proof is similar as above, using the fact that
$$\mathbb{E}\Big(\sum_{i\in \mathbb{N}}e^{-C_A\vert ih_K\vert} \big( N^K_i(0)\big)^2 \Big)\leq C(T,K),$$ which results from \eqref{eq:initial linear growth2} and \eqref{eq:initial control supcrtitical case}. That completes the proof of \eqref{eq:moments control}.
\end{proof}
\begin{pro}
Under Assumptions A-C, Equation \eqref{eq:mean dynamics system} has a unique solution $n^K\in C^1(\R^+,\widetilde{\ell}^1_A(\Z)),$
where \begin{equation}
\label{eq: new space}
\widetilde{\ell_A}^1(\Z):=\{v\in \R^\Z,~\sum_{i\in \Z}e^{-C_A\vert ih_K\vert}\vert v_i\vert<+\infty\},
\end{equation}
and $C_A:=2A+1,$ introduced in Proposition 4.1.
\end{pro}
See Appendix B for the proof.\\
The next proposition shows the convergence of the Hopf-Cole transformation $u^K$ to the viscosity solution of the Hamilton-Jacobi equation, under Assumptions A-C.

\begin{pro}
    Under Assumptions A-C, $(\widetilde{u}^K)_K$ converges locally uniformly to the unique Lipschitz-continuous viscosity solution of the Hamilton-Jacobi equation:
\begin{equation}
\label{eq:p-HJ}
    \begin{cases}
        \partial_t u(t,x)=b(x)-d(x)+p\int_{\R}G(y)e^{\nabla u(t,x).y}dy,~~(t,x)\in(0,+\infty)\times \R,\\
        u(0,.)=u^0(.).
    \end{cases}
\end{equation}
\end{pro}
See Appendix B for the proof.
\subsection{Convergence result in the supercritical regime}
\begin{thm}
Under Assumptions A-C, the sequence of stochastic processes $(\widetilde{\beta}^K)_K$ converges in probability locally uniformly to the unique Lipschitz-continuous viscosity solution of the equation 
\begin{equation*}
    \begin{cases}
        \partial_t u(t,x)=b(x)-d(x)+p\int_{\R}G(y)e^{\nabla u(t,x).y}dy,~~(t,x)\in(0,+\infty)\times \R,\\
        u(0,.)=u^0(.).
    \end{cases}
\end{equation*}
\end{thm}
This theorem  gives a probabilistic derivation of the classical Hamilton-Jacobi equation in a uniformly supercritical dynamics. 
\begin{pro}
   Let $T>0.$ For any $(t,i)\in[0,T]\times \Z,$ we have 
    \begin{equation}
        \label{eq:uniform control}
    \Esp\left(\frac{N^K_i(t\log{K})-n^K_i(t\log{K})}{n^K_i(t\log{K})}\right)^2 \leq \frac{C(T)}{K^{a/2}},
    \end{equation}
    where $ C(T)$ is a positive constant increasing on $T,$ but independent on $K.$\\
    Moreover, for all $\eta>0$ and $T,D>0,$ we have 
    \begin{equation}
    \label{eq:proba convergence ineq}
      \mathbb{P}\Bigg(\sup_{\{(t,i)\in[0,T]\times \Z/ i\delta_K\in [-D,D]\}} \Big\vert \frac{N^K_i(t\log{K})-n^K_i(t\log{K})}{n^K_i(t\log{K})}\Big\vert >\eta\Bigg)\leq\frac{C(T,D)}{\eta^2 \delta_KK^{a/2}},
\end{equation}
    where $C(T,D)$ is a positive constant depending on $T,D$ and increasing on $T,$ but independent on $K.$ Hence, on a logarithmic time scale, $N^K/n^K$ converges in probability to
$1$ locally uniformly in time and space.
\end{pro}   
\begin{proof}
    Let $T>0.$ For all $(t,i)\in [0,T]\times\Z,$ we define $S^K_i(t)=\frac{N^K_i(t)-n^K_i(t)}{n^K_i(t)}.$ Using Itô's lemma and the semi-martingale decomposition  \eqref{eq:s-m}, which is well-defined by \eqref{eq:moments control}, we have for all $0<s<t$  
\begin{align*} 
S^K_i(t)=S^K_i(s)+\int_{s}^{t}\sum_{l\in \Z}p h_KG(lh_K)\frac{n^K_{l+i}(\tau)}{n^K_i(\tau)} (S^K_{l+i}(\tau)-S^K_i(\tau))d\tau +\int_{s}^{t}\frac{dM^K_i(\tau)}{n^K_i(\tau)}.
\end{align*}
    Then 
    \begin{equation}
    \label{eq:martingal form}
    \begin{aligned}
  \big(S^K_i(t)\big)^2=\big(S^K_i(s)\big)^2+  \int_{s}^{t}\sum_{l\in \Z}ph_KG(lh_K)\frac{n^K_{l+i}(\tau)}{n^K_i(\tau)} 2S^K_i(\tau) (S^K_{l+i}(\tau)-S^K_i(\tau))ds&+\int_{s}^{t}\frac{d \langle M^K_i\rangle_\tau}{\big(n^K_i(\tau)\big)^2}\\
  +\widetilde{M}^K_i(t)-\widetilde{M}^K_i(s),
  \end{aligned}
    \end{equation}
    where $t\mapsto \widetilde{M}^K_i$ is a martingale, which is a consequence of \eqref{eq:moments control}. We define, $Y^K_i(t)=\Esp\Big(\big (S^K_i(t)\big)^2\Big),$ which is also well-defined by \eqref{eq:moments control}. By Young's inequality, we obtain that
    \begin{align*}
   Y^K_i(t) \leq Y^K_i(s)+  \int_{s}^{t}\sum_{l\in \Z}ph_KG(lh_K)\frac{n^K_{l+i}(\tau)}{n^K_i(\tau)}(Y^K_{l+i}(\tau)-Y^K_i(\tau))d\tau+\int_{s}^{t}\frac{d \Esp( \langle M^K_i\rangle_\tau)}{\big(n^K_i(\tau)\big)^2}.
    \end{align*}
Furthermore, from \eqref{eq: spatial lischitz estimates} the spatial Lipschitz bound of $u^K,$ is global on $[0,T]\times Z,$ and $T\mapsto C_{Lip}(T)$ is increasing. We obtain
\begin{align*}
    \int_{s}^{t}\frac{d \Esp( \langle M^K_i\rangle_\tau)}{\big(n^K_i(\tau)\big)^2}&=\int_{s}^{t}\frac{1}{n^K_i(\tau)}\Big (b(i\delta_K)+d(i\delta_K)+\sum_{l\in \Z}p h_KG(lh_K)\frac{n^K_{l+i}(\tau)}{n^K_i(\tau)} d\tau\Big)\\
    &\leq \int_{s}^{t}\frac{C(T/\log{K})d\tau}{n^K_i(\tau)},
\end{align*}
where $T\mapsto C(T)$ is increasing. 
Then by \eqref{eq:large dynamics}, we deduce that 
\begin{equation}
\label{eq:maximum point form}
    \frac{d Y^K_i(t)}{dt} \leq \sum_{l\in \Z} ph_KG(lh_K)\frac{n^K_{l+i}(t)}{n^K_i(t)}(Y^K_{l+i}(t)-Y^K_i(t))+\frac{C(T/\log{K})}{K^a}.
\end{equation}
Therefore, by the maximum principle (see Appendix B for the proof), and \eqref{eq:initial control supcrtitical case}, we deduce \eqref{eq:uniform control}.
Moreover, let $\eta>0,$ we have by the submartingale inequality and from \eqref{eq:martingal form}, that for all $i\in \Z$
\begin{equation*}
    \begin{aligned}
    \mathbb{P}(\sup_{t\in[0,T]}\vert \widetilde{M}^K_i(t)\vert>\eta)\leq &\frac{1}{\eta}\mathbb{E}(\vert \widetilde{M}^K_i(T)\vert)\\
    &\leq \frac{1}{\eta}\Big( \mathbb{E}\big((S^K_i(T))^2\big)+\mathbb{E}\big((S^K_i(0))^2\big) \\
    &~+\int_{0}^{T}\sum_{l\in \Z}ph_KG(lh_K)\frac{n^K_{l+i}(s)}{n^K_i(s)} 2 \Esp\big(\vert S^K_i(s) (S^K_{l+i}(s)-S^K_i(s))\vert\Big) ds\\
    &~+\int_{0}^{T}\frac{d \Esp( \langle M^K_i\rangle_s)}{\big(n^K_i(s)\big)^2}\Big).
\end{aligned}
\end{equation*}
Moreover, we have 
\begin{equation}
\begin{aligned}
    \sum_{l\in \Z}ph_KG(lh_K)&\frac{n^K_{l+i}(s)}{n^K_i(s)}\Esp\Big(2\vert S^K_i(s) (S^K_{l+i}(s)-S^K_i(s))\vert\Big)\\
    &\leq p\sum_{l\in\Z} h_KG(lh_K)e^{C_{Lip}(T/\log{K})\vert lh_K\vert} (3Y^K_i(s)+Y^K_{l+i}(s)).
    \end{aligned}
\end{equation}
Thus, from \eqref{eq:uniform control} we deduce that 
\begin{equation}
\label{eq:martingal control}
    \mathbb{P}(\sup_{t\in[0,T]}\vert \widetilde{M}^K_i(t)\vert>\eta)\leq \frac{C(T/\log{K})}{\eta K^{a/2}}.
\end{equation}
Therefore, from \eqref{eq:martingal form} and \eqref{eq:martingal control}, we obtain for $K$  large enough that
\begin{equation}
    \begin{aligned}
        \mathbb{P}(\sup_{t\in[0,T]} (S^K_i(t))^2>\eta)&\leq   \mathbb{P}\big( (S^K_i(0))^2>\eta/4\big)+\mathbb{P}\big(\sup_{t\in[0,T]}\vert \widetilde{M}^K_i(t)\vert>\eta/4\big) \\
        ~&+\mathbb{P}\Big( \int_{0}^{T}\sum_{l\in \Z}ph_KG(lh_K)\frac{n^K_{l+i}(s)}{n^K_i(s)} 2\vert S^K_i(s) (S^K_{l+i}(s)-S^K_i(s))\vert ds>\eta/4\Big)\\
        &\leq \frac{C(T/\log{K})}{\eta K^{a/2}}.
    \end{aligned}
\end{equation}
This implies \eqref{eq:proba convergence ineq}, and hence  $N^K/n^K$ converges in probability to $1,$ on a logarithmic time scale. We conclude the proof of Proposition 4.5.

\end{proof}
\subsection{Proof of Theorem 4.4}

Let us prove the convergence toward the Hamilton-Jacobi equation. By similar computations as in the proof of Theorem 3.1, and by  \eqref{eq:proba convergence ineq}, we deduce that
\begin{align*}
    &\mathbb{P}\Bigg( \sup_{\{(t,x)\in [0,T]\times [-D,D]\}}\big \vert \widetilde{\beta}^K(t,x)-\widetilde{u}^K(t,x)\big\vert>\eta\Big)\\
    &\leq \mathbb{P}\Bigg( \sup_{\{(t,i)\in [0,T]\times \Z/ i\delta_K\in[-2D,2D]\}}\big \vert \beta^K_i(t)-u^K_i(t)\big\vert>\eta\Big)\\
    &=\mathbb{P}\Bigg( \sup_{(t,i\delta_K)\in [0,T]\times [-2D,2D]} \frac{1}{\log{K}}\Big\vert \log{\Big(1+\frac{N^K_i(t\log{K})-n^K_i(t\log{K})}{n^K_i(t\log{K})}\Big)}\Big\vert>\eta\Big)\\
    &\leq (\frac{1}{\eta^2\log^2{K}}+4)\frac{C(T,2D)}{\delta_K K^{a/2}}.
\end{align*}
Moreover, from Proposition 4.3, $u^K$ converges to the viscosity solution of \eqref{eq:p-HJ}, we conclude the proof of Theorem 4.4.
\newpage 
\section*{Appendix A}
\textbf{Proof of the maximum principle for \eqref{eq:maximum principle form1}:} Recall that $K$ is fixed in this proof.\\
Let $T>0.$ We recall that $ Y^K_i(t)=\mathbb{E}\Big(\frac{(N^K_i(t)-n^K_i(t))^2}{n^K_i(t)}\Big).$ We start by bounding the growth of $Y^K_i$ when $\vert i\vert \rightarrow +\infty.$
We have, $$Y^K_i(t)\leq 2 \frac{\mathbb{E}\big((N^K_i(t))^2\big)}{n^K_i(t)}+ 2n^K_i(t).$$ 
From \eqref{eq:second order moment1}, we have
$$\mathbb{E}\big((N^K_i(t))^2\big)\leq C(T,K),\text{and}~~n^K_i(t)\leq \sqrt{C(T,K)},~~\forall t\in[0,T].$$
Furthermore, from \eqref{eq:local bounds} we have that
$n^K_i(t)\geq e^{-L\vert ih_K\vert+A_4\log{K}+A_5t}.$ Therefore, we deduce that
\begin{equation}
\label{eq:behavior at infinity}
    Y^K_i(t)\leq \widetilde{C}(T,K)(e^{L\vert ih_K\vert }+1).
\end{equation}
Our goal is to deduce from \eqref{eq:maximum principle form1} that 
\begin{equation}
\label{eq:pm}
    (Y^K_i(t)-C(T/\log{K})t)e^{-(\alpha +o(1))t} -C\leq 0,
    \end{equation}
where $C$ is given by \eqref{eq:uniform initial control1}.
To this end, let us define for all $(t,i)\in[0,T]\times\mathbb{Z},$ and for some positive constant $D$ to be chosen latter  $$\widetilde{Y}^K_i(t):=\Big(\big(Y^K_i(t)-C(T/\log{K})t\big)e^{-(\alpha +o(1))t}-C\Big) e^{-(L+1)\vert lh_K\vert}e^{-Dt}.$$
By contradiction, we assume that 
\begin{equation}
    M=\sup_{(t,i)\in [0,T]\times\mathbb{Z}}\widetilde{Y}^K_i(t)>0.
\end{equation}
From \eqref{eq:behavior at infinity}, $\widetilde{Y}^K_i(t)$ vanishes as $\vert i\vert \rightarrow+\infty.$ Let 
$(t_K,i_K)$ be a maximum point of $\widetilde{Y^K}$ in $[0,T]\times \mathbb{Z}.$ By \eqref{eq:initial control supcrtitical case} we have that $t_K>0.$ Thus, it follows from \eqref{eq:maximum principle form1}  at $(t_K,i_K)$ that
\begin{align*}
    \frac{d}{dt}\widetilde{Y}^K_{i_{K}}(t_{K})&\leq -D\widetilde{Y}^K_{i_K}(t_K)+p\sum_{l\in \mathbb{Z}}h_KG(lh_K)\frac{n^K_{l+i_K}(t_K)}{n^K_{i_K}(t_K)}(e^{(L+1)(\vert (l+i_K)h_K\vert -\vert i_K h_K\vert)}\widetilde{Y}^K_{l+i_K}(t_K)-\widetilde{Y}^K_{i_K}(t_K))\\
    &\leq-DM +p\sum_{l\in \mathbb{Z}}h_KG(lh_K)\frac{n^K_{l+i_K}(t_K)}{n^K_{i_K}(t_K)}(e^{(L+1)\vert lh_K\vert} M-M)\\
    &\leq \Bigg(-D+p \sum_{l\in \mathbb{Z}}h_KG(lh_K)(e^{(L+1+C_{Lip}(T/\log{K}))\vert lh_K\vert}\Bigg)M.
\end{align*}
Then, for $D>p \sum_{l\in \mathbb{Z}}h_KG(lh_K)(e^{(L+1+C_{Lip}(T/\log{K}))\vert lh_K\vert},$ we have 
$$\frac{d}{dt}\widetilde{Y}^K_{i_K}(t_{K})<0.$$ This is in contradiction with the fact that $t_K>0.$ Hence \eqref{eq:pm} is proved. \\
\subsection*{Appendix B}
\textbf{Proof of Proposition 4.2}
The proof follows the classical Picard iteration method. The novelty here is that the initial condition is not integrable. We introduce a new functional space in which existence and uniqueness hold. For more details about the existence and uniqueness of such a discrete system, we refer to Appendix A in \cite{J}.\\
We define  $B^K=(a^K_{i,j})_{i,j\in \mathbb{Z}}$ the infinite real matrix where 
$a^K_{i,j}=ph_K G((j-i)h_K),$ $D^K$ the infinite diagonal matrix whose diagonal elements are $b(i\delta_K)-d(i\delta_K)$ for all $i\in \mathbb{Z},$ and we define the infinite vector  $n^K=(n^K_i)_{i\in \mathbb{Z}}.$ We write Equation \eqref{eq:mean dynamics system} as
\begin{equation}
    \frac{d}{dt}n^K(t)=(D^K+B^K)n^K(t).
\end{equation} 
Let $T>0.$ We consider the following closed subset of $C( [0,T],\widetilde{\ell}^1_A(\mathbb{Z})):$
\begin{equation}
\mathcal{H}:=\Big\{n\in  C( [0,T],\widetilde{\ell}^1_A(\mathbb{Z})),~ \Vert n(t)\Vert_{\widetilde{\ell}^1_A(\mathbb{Z})}\leq C(t) \Vert n(0)\Vert_{\widetilde{\ell}^1_A(\mathbb{Z})},~\forall t\in [0,T]\Big\},
\end{equation}
where  $$\Vert n\Vert_{L^{\infty}([0,T],\widetilde{\ell}^1_A(\mathbb{Z}))}:=\sup_{t\in[0,T]}\sum_{i\in \mathbb{Z}}e^{-C_A\vert ih_K\vert} \vert n_i(t)\vert,$$ and $C(t):=1+e^{(\overline{b}+\overline{d}+2\overline{p})t}.$ Note that $C( [0,T],\widetilde{\ell}^1_A(\mathbb{Z}))$ endowed with the norm $\Vert .\Vert_{L^{\infty}([0,T],\widetilde{\ell}^1_A(\mathbb{Z}))},$ is a Banach space. Let us now define the mapping 
 \begin{align*}
 \Phi:& \mathcal{H}\mapsto \mathcal{H}\\
 &n \mapsto \Phi(n),
 \end{align*}
where $\Phi(n)$ is the unique solution of \begin{align*}
    \begin{cases}
       \frac{d}{dt}\Phi(n)(t)=(D^K+B^K)n(t)\\
       \Phi(n)(0)=n^K(0).
    \end{cases}
\end{align*} 
Note that from \eqref{eq:initial linear growth2} there exists a positive constant $C(K)$ such that $\Vert n^K(0)\Vert_{\widetilde{\ell}^1_A(\mathbb{Z})}\leq C(K).$ Then, the proof follows by simple computations proving that, for a small time $T>0,$ $\Phi(\mathcal{H})\subset \mathcal{H},$ and it is a contraction in $\mathcal{H}$ with the norm $\Vert .\Vert_{L^{\infty}([0,T],\widetilde{\ell}^1_A(\mathbb{Z}))}.$ We conclude the existence and uniqueness until this small time $T,$ and by classical iteration arguments, we conclude the existence and uniqueness of the global solution. 
\subsection*{Proof of Proposition 4.3}
The proof of this result follows similar arguments to the proof of Theorem 3.3. Note that from  \eqref{eq:initial linear growth2}, and by adapting the proof of the comparison principle in Proposition 3.4 in \cite{J}, we obtain for some positive constant $C,$ that $u^K_i(t)\leq A\vert i\delta_K\vert +B+Ct.$ Moreover, from \eqref{eq:large dynamics}, we have $u^K_i(t)\geq a.$ We deduce that, $u^K$ is locally uniformly bounded. The spatial Lipschitz bound is given by Proposition 4.2 in \cite{J}, and the Lipschitz bound in time can be deduced as in the proof of  Theorem 3.3. Then, by the Arzelà–Ascoli
 theorem, we conclude the compactness of $(\widetilde{u}^K)_K$, and classical arguments  yield the convergence along subsequences to the Hamilton-Jacobi equation. The uniqueness of such  Hamilton-Jacobi equations with an initial condition of linear growth is well-known (see for instance \cite{G}). Hence, the local uniform convergence of the full sequence $(\widetilde{u}^K)_K$ to the unique viscosity solution of \eqref{eq:p-HJ}.
\subsection*{Proof of the maximum principle of \eqref{eq:maximum point form}}
The only change from the proof given in Appendix A is the behavior of $Y^K$ at infinity in space. In this case, it is easy to see from \eqref{eq:moments control} that 
$\max\Big(n^K_i(t),\mathbb{E}\big((N^K_i(t))^2\big)\Big)\leq C(T,K)e^{C_A\vert ih_K\vert}.$ Moreover, since $ n^K_i(t)\geq K^a,$ we  deduce the exponential growth of $Y^K$ for fixed $K,$ i,e  $$Y^K_i(t)\leq C(T,K)e^{C\vert ih_K\vert},$$ 
where $C$ is a positive constant. Then, the remainder of the proof is exactly as in Appendix A.\\

\textbf{Acknowledgements:} The author thanks Nicolas Champagnat and Sylvie Méléard for their  comments and careful reading. This work is funded by the European Union (ERC, SINGER, 101054787). Views and opinions expressed are however those of the author only and do not necessarily reflect those of the European
Union or the European Research Council. Neither the European Union nor the granting authority
can be held responsible for them. This work has also been supported by the Chair "Modélisation Mathématique et Biodiversité" of Veolia Environnement-École Polytechnique-Museum National
d’Histoire Naturelle-Fondation X.
\newpage
\nocite{*}
\bibliographystyle{plain}
\bibliography{ref.bib}

\end{document}